\documentclass[12pt]{amsart}
\usepackage{amsmath,amssymb,amsbsy,amsfonts,latexsym,amsopn,amstext,cite,
                                               amsxtra,euscript,amscd,bm,mathabx}
\usepackage{url}
\usepackage[colorlinks,linkcolor=blue,anchorcolor=blue,citecolor=blue,backref=page]{hyperref}
\usepackage{color}
\usepackage{graphics,epsfig}
\usepackage{graphicx}
\usepackage{float} 
\usepackage[english]{babel}
\usepackage{mathtools}
\usepackage{todonotes}
\usepackage{url}
\usepackage[colorlinks,linkcolor=blue,anchorcolor=blue,citecolor=blue,backref=page]{hyperref}
\DeclareMathAlphabet{\mathmybb}{U}{bbold}{m}{n}

\usepackage[norefs,nocites]{refcheck}

\hypersetup{breaklinks=true}

\usepackage[norefs,nocites]{refcheck}
\usepackage[english]{babel}
\begin{document}

\newtheorem{thm}{Theorem}
\newtheorem{lem}[thm]{Lemma}
\newtheorem{claim}[thm]{Claim}
\newtheorem{cor}[thm]{Corollary}
\newtheorem{prop}[thm]{Proposition} 
\newtheorem{definition}[thm]{Definition}
\newtheorem{rem}[thm]{Remark} 
\newtheorem{question}[thm]{Open Question}
\newtheorem{conj}[thm]{Conjecture}
\newtheorem{prob}{Problem}

\newtheorem{lemma}[thm]{Lemma}

\newcommand{\GL}{\operatorname{GL}}
\newcommand{\SL}{\operatorname{SL}}
\newcommand{\lcm}{\operatorname{lcm}}
\newcommand{\ord}{\operatorname{ord}}
\newcommand{\Op}{\operatorname{Op}}
\newcommand{\Tr}{\operatorname{Tr}}
\newcommand{\Nm}{\operatorname{Nm}}

\numberwithin{equation}{section}
\numberwithin{thm}{section}
\numberwithin{table}{section}

\def\vol {{\mathrm{vol\,}}}
\def\squareforqed{\hbox{\rlap{$\sqcap$}$\sqcup$}}
\def\qed{\ifmmode\squareforqed\else{\unskip\nobreak\hfil
\penalty50\hskip1em\null\nobreak\hfil\squareforqed
\parfillskip=0pt\finalhyphendemerits=0\endgraf}\fi}

\def \ss{\mathsf{s}} 

\def \balpha{\bm{\alpha}}
\def \bbeta{\bm{\beta}}
\def \bgamma{\bm{\gamma}}
\def \blambda{\bm{\lambda}}
\def \bchi{\bm{\chi}}
\def \bphi{\bm{\varphi}}
\def \bpsi{\bm{\psi}}
\def \bomega{\bm{\omega}}
\def \btheta{\bm{\vartheta}}

\newcommand{\bfxi}{{\boldsymbol{\xi}}}
\newcommand{\bfrho}{{\boldsymbol{\rho}}}

\def\Kab{\sfK_\psi(a,b)}
\def\Kuv{\sfK_\psi(u,v)}
\def\SaUV{\cS_\psi(\balpha;\cU,\cV)}
\def\SaAV{\cS_\psi(\balpha;\cA,\cV)}

\def\SUV{\cS_\psi(\cU,\cV)}
\def\SAB{\cS_\psi(\cA,\cB)}

\def\Kmnp{\sfK_p(m,n)}

\def\KKap{\cH_p(a)}
\def\KKaq{\cH_q(a)}
\def\KKmnp{\cH_p(m,n)}
\def\KKmnq{\cH_q(m,n)}

\def\Klmnp{\sfK_p(\ell, m,n)}
\def\Klmnq{\sfK_q(\ell, m,n)}

\def \SALMNq {\cS_q(\balpha;\cL,\cI,\cJ)}
\def \SALMNp {\cS_p(\balpha;\cL,\cI,\cJ)}

\def \SACXMQX {\fS(\balpha,\bzeta, \bxi; M,Q,X)}

\def\SAMJp{\cS_p(\balpha;\cM,\cJ)}
\def\SAMJq{\cS_q(\balpha;\cM,\cJ)}
\def\SAqMJq{\cS_q(\balpha_q;\cM,\cJ)}
\def\SAJq{\cS_q(\balpha;\cJ)}
\def\SAqJq{\cS_q(\balpha_q;\cJ)}
\def\SAIJp{\cS_p(\balpha;\cI,\cJ)}
\def\SAIJq{\cS_q(\balpha;\cI,\cJ)}

\def\RIJp{\cR_p(\cI,\cJ)}
\def\RIJq{\cR_q(\cI,\cJ)}

\def\TWXJp{\cT_p(\bomega;\cX,\cJ)}
\def\TWXJq{\cT_q(\bomega;\cX,\cJ)}
\def\TWpXJp{\cT_p(\bomega_p;\cX,\cJ)}
\def\TWqXJq{\cT_q(\bomega_q;\cX,\cJ)}
\def\TWJq{\cT_q(\bomega;\cJ)}
\def\TWqJq{\cT_q(\bomega_q;\cJ)}

 \def \xbar{\overline x}
  \def \ybar{\overline y}

\def\cA{{\mathcal A}}
\def\cB{{\mathcal B}}
\def\cC{{\mathcal C}}
\def\cD{{\mathcal D}}
\def\cE{{\mathcal E}}
\def\cF{{\mathcal F}}
\def\cG{{\mathcal G}}
\def\cH{{\mathcal H}}
\def\cI{{\mathcal I}}
\def\cJ{{\mathcal J}}
\def\cK{{\mathcal K}}
\def\cL{{\mathcal L}}
\def\cM{{\mathcal M}}
\def\cN{{\mathcal N}}
\def\cO{{\mathcal O}}
\def\cP{{\mathcal P}}
\def\cQ{{\mathcal Q}}
\def\cR{{\mathcal R}}
\def\cS{{\mathcal S}}
\def\cT{{\mathcal T}}
\def\cU{{\mathcal U}}
\def\cV{{\mathcal V}}
\def\cW{{\mathcal W}}
\def\cX{{\mathcal X}}
\def\cY{{\mathcal Y}}
\def\cZ{{\mathcal Z}}
\def\Ker{{\mathrm{Ker}}}

\def\NmQR{N(m;Q,R)}
\def\VmQR{\cV(m;Q,R)}

\def\Xm{\cX_{p,m}}

\def \A {{\mathbb A}}
\def \B {{\mathbb A}}
\def \C {{\mathbb C}}
\def \F {{\mathbb F}}
\def \G {{\mathbb G}}
\def \L {{\mathbb L}}
\def \K {{\mathbb K}}
\def \PP {{\mathbb P}}
\def \Q {{\mathbb Q}}
\def \R {{\mathbb R}}
\def \Z {{\mathbb Z}}
\def \fS{\mathfrak S}
\def \fB{\mathfrak B}

\def\GL{\operatorname{GL}}
\def\SL{\operatorname{SL}}
\def\PGL{\operatorname{PGL}}
\def\PSL{\operatorname{PSL}}
\def\li{\operatorname{li}}
\def\sym{\operatorname{sym}}

\def\Mob{M{\"o}bius }

\def\fF{\EuScript{F}}
\def\M{\mathsf {M}}
\def\T{\mathsf {T}}

\def\e{{\mathbf{\,e}}}
\def\ep{{\mathbf{\,e}}_p}
\def\eq{{\mathbf{\,e}}_q}

\def\\{\cr}
\def\({\left(}
\def\){\right)}

\def\<{\left(\!\!\left(}
\def\>{\right)\!\!\right)}
\def\fl#1{\left\lfloor#1\right\rfloor}
\def\rf#1{\left\lceil#1\right\rceil}

\def\Tr{{\mathrm{Tr}}}
\def\Nm{{\mathrm{Nm}}}
\def\Im{{\mathrm{Im}}}

\def \oF {\overline \F}

\newcommand{\pfrac}[2]{{\left(\frac{#1}{#2}\right)}}

\def \Prob{{\mathrm {}}}
\def\e{\mathbf{e}}
\def\ep{{\mathbf{\,e}}_p}
\def\epp{{\mathbf{\,e}}_{p^2}}
\def\em{{\mathbf{\,e}}_m}

\def\Res{\mathrm{Res}}
\def\Orb{\mathrm{Orb}}

\def\vec#1{\mathbf{#1}}
\def \va{\vec{a}}
\def \vb{\vec{b}}
\def \vh{\vec{h}}
\def \vk{\vec{k}}
\def \vs{\vec{s}}
\def \vu{\vec{u}}
\def \vv{\vec{v}}
\def \vz{\vec{z}}
\def\flp#1{{\left\langle#1\right\rangle}_p}
\def\T {\mathsf {T}}

\def\sfG {\mathsf {G}}
\def\sfK {\mathsf {K}}

\def\mand{\qquad\mbox{and}\qquad}

\title[On bounds on Dedekind sums for subgroups]
{Pointwise and correlation bounds on Dedekind sums over small subgroups}

\author[B.  Borda]{Bence Borda}
\address{Institute of Analysis and Number Theory, Technische Universit{\"a}t Graz, 
Steyrergasse 30, 8010 Graz, Austria}
\email{ borda@math.tugraz.at}

\author[M. Munsch]{Marc Munsch}
\address{Institut Camille Jordan, Universit{\'e} Jean Monnet, 23, Rue du docteur Paul Michelon,
42 023 Saint-Etienne Cedex 2, France}
\email{marc.munsch@univ-st-etienne.fr}

\author[I. E. Shparlinski] {Igor E. Shparlinski}
\address{School of Mathematics and Statistics, University of New South Wales, Sydney, NSW 2052, Australia}
\email{igor.shparlinski@unsw.edu.au}

\begin{abstract}
We obtain new bounds, pointwise and on average, for Dedekind sums $\ss(\lambda,p)$ modulo a 
prime $p$ with $\lambda$ of small multiplicative order $d$ modulo $p$.  Assuming the infinitude of Mersenne primes,
the range of our results is optimal. Moreover, we relate high moments of $L(1,\chi)$ over subgroups of characters to some correlations of Dedekind sums, and use recent results of the second and third authors to study these correlations.
\end{abstract}

\maketitle

\section{Introduction}

\subsection{Dedekind sums and moments of $L$-functions}
Given two integers $a$ and $b$ with 
\[
b \ge 1 \mand \gcd(a,b)=1,
\]
we define the {\it Dedekind sum\/} 
\[
\ss(a,b) = \sum_{c=1}^{b-1} \<\frac{c}{b}\>  \<\frac{ac}{b}\>, 
\]
where 
\[
\left(\!\left(\xi \right)\!\right) = \begin{cases}
 \{\xi\} - 1/2, & \text{if } \xi\in \R\backslash \Z,\\
0, & \text{if } \xi\in   \Z.
\end{cases}  
\]  

Historically, Dedekind sums appeared in the context of modular forms to describe the 
transformation formula for the Dedekind eta function~\cite{RGr}.
Since the seminal works of 
Vardi~\cite{Var1,Var2}, it is known that the distribution of the ratio  $\ss(a,b)/\log b$ converges to the Cauchy distribution when averaged over pairs $\gcd(a,b)=1$ and $1\leq b \leq B$, with $B\to \infty$,  see 
also~\cite{BaSh,Hick}. In particular, Dedekind sums are relatively small on average, but can also seldom take very large values which are responsible for the asymptotic of the moments, see, for instance,~\cite{CFKS, Gir1, Gir2} and references therein. Both the study of the distribution of values of Dedekind sums~\cite{ABH,BD} as well as the study of restricted averages of 
Dedekind sums~\cite{MST,MST2}  regained a lot of interest in recent years, in particular due to their applications to
the  analysis of some number theoretic algorithms. Let us also mention that Lemke Oliver and Soundararajan~\cite{L-OSound} have discovered a very interesting link between  Dedekind sums and the distribution of primes.  

We first observe that by the Cauchy--Schwarz inequality,   we have the following 
trivial bound:
\[
|\ss(a,b)| \le \ss(1,b) = \frac{(b-1)(b-2)}{12b} < \frac{b}{12} .
\]
Here we are interested in non-trivial bounds on the sums $\ss(\lambda,p)$ depending on the multiplicative order 
of $\lambda$  modulo a prime $p$. 
In particular, this question is   motivated by the connection 
 with the second moments of Dirichlet $L$-functions $L(1, \chi)$ over 
multiplicative characters $\chi$ of a given order, as established in~\cite{Lou1,Lou2,LoMu1,LoMu2}. A link 
between the size of $\ss(a,b)$ and the order of $a$ modulo $b$ was first established 
by Louboutin in \cite[Lemma~6]{Lou1}, which implies that for an odd 
\begin{equation}\label{eq:Bad b}
b= (a^d -1)/(a-1)
\end{equation}
 with an integer $a\ne 0, \pm 1$ and a prime $d$, we have
 $$
%%\begin{equation}\label{asymp-Mersenne}
\ss(a,b) = \frac{(b-1)(b-a^2 -1)}{12 ab}  = O\(b^{1-1/(d-1)}\).
%%\end{equation}
$$
It is easy to see that $d$ is the order of $a$ modulo $b$. 

More precisely, let  $ \F_p^*$  be 
the  multiplicative group of the finite field  $\F_p$  of $p$ elements, and let $\cX_p$ be
the group of multiplicative characters of  $\F_p^*$.  

Given an even integer $m$ with $m\mid p-1$  and  a multiplicative subgroup $\cG_m\subseteq \F_p^*$ of index $m$ and thus of order $(p-1)/m$, we denote by $\cX_{p,m}$ 
the group of multiplicative characters of  $\F_p^*$  which are trivial on $\cG_m$.
Furthermore, assuming that $(p-1)/m$ is odd, we consider the set of $m/2$ odd 
characters in $\cX_{p,m}$, that is, the set
\[
\cX_{p,m}^{-} = \left\{\chi \in \cX_{p,m}:~\chi(-1)=-1\right\}.
\] 
We also set 
\[
\cX_{p}^{-} = 
\cX_{p,p-1}^{-} .
\]  
We now define the moments
\[
M_{\nu}^{-}(p,m)
=\frac{2}{m} \sum_{\chi \in \cX_{p,m}^{-} } |L(1, \chi)|^{\nu},  \qquad \nu =1,2, \ldots,
\]
where $L(s, \chi) $ is the Dirichlet $L$-function associated with a multiplicative character $\chi \in \cX_p$.

By~\cite[Chapter~4]{Wash}, there is a close connection between the second moment  $M_2^{-}(p, m) $ and the
relative class number of the imaginary subfield of the cyclotomic field $\Q\(\exp(2 \pi i/p)\)$ of  degree $m$ over $\Q$.  
In turn,  by a result of Louboutin~\cite[Theorem~2]{Lou1},  we have 
\begin{equation}
\label{eq:M & D}
M_{2}^{-}(p, m )  = \frac{\pi^2}{6}\(1 - \frac{3}{p} + \frac{2}{p^2} +  \frac{2}{p} S(p, m)\), 
\end{equation}
where 
\[
S(p,m) = \sum_{\substack{\lambda\in \cG_m\\\lambda \ne 1}} \ss(\lambda,p).
\]

Hence the relation~\eqref{eq:M & D} and the above connection to the class number motivates studying the size of $\ss(\lambda,p)$ on elements $\lambda \in \F_p^*$ of given order as well as the size on average over a subgroup of $\F_p^*$.  Moreover, using asymptotic  formulas for $M_2^{-}(p, m )$
 from~\cite{MuSh}, one immediately  obtains upper bounds for $ S(p, m)$. Here we derive several 
 new results which go beyond this link.

\subsection{Summary of  our results and methods}  
We first improve a pointwise bound from~\cite{LoMu2} on Dedekind sums $\ss(\lambda,p)$,  provided that $\lambda$ 
is of small multiplicative order in $\F_p^*$, see  Corollary~\ref{cor:D-single} below. It could be in principle used  to improve the error term in the asymptotic formula for $M_2^{-}(p, m )$, provided that $m$ is not too small. However, the result is now superseded by a much more general result of the second and third authors~\cite[Theorem~2.1]{MuSh} obtained via different methods. It may have other applications, though.
 Under the assumption of the infinitude  
of {\it Mersenne primes\/}, the ranges where these results provide non-trivial estimates are optimal, see Remark~\ref{rem:limit}.

Next, we generalise the formula~\eqref{eq:M & D} to higher moments relating $M_{2k}^-(p,m)$ to certain correlations between Dedekind sums (see, for example, Theorem~\ref{thm-formulahigher} below), and use the results of~\cite{MuSh} to get information on these sums. 

 Our methods involve  the connection between Dedekind sums and {\it continued fractions\/} 
established independently by Barkan~\cite{Bar} and Hickerson~\cite{Hick} (see also a more general result of Knuth~\cite{Knuth}). 

\subsection{Notation and conventions}

We adopt the Vinogradov notation $\ll$,  that is,
\[A\ll B~\Longleftrightarrow~B\gg A~\Longleftrightarrow ~A=O(B)~\Longleftrightarrow~|A|\le cB\]
for some constant $c>0$ which sometimes, where obvious, may depend on the 
integer parameter $k\ge 1$, and is absolute otherwise. 

For a finite set $\cS$, we use $\# \cS$ to denote its cardinality. 
We also write $\|x\|$ to denote the distance from a real $x$ to the closest integer and, 
as usual, $\varphi(d)$ denotes the Euler function.

\section{Bounds on Dedekind sums} 

\subsection{Pointwise bounds on Dedekind sums}
We are now ready to present an improvement of~\cite[Theorem~3.1]{LoMu2}, 
which gives the bound  
\begin{equation}\label{LMbound} \ss(\lambda,p)  \ll p^{1-1/\varphi(d)} (\log p)^2 \end{equation}
provided that  $\lambda$ is of order $d\geq 3$ in $\F_p^*$. 
Here we remove the factor  $(\log p)^2$, which gives an optimal range with respect to $d$ where 
one can hope for getting a non-trivial bound on $\ss(\lambda,p)$, see Remark~\ref{rem:limit}.

We start with the following result about continued fractions of rationals with prime denominators, 
which might be of independent interest.

\begin{thm}
\label{thm:CF}  
Let $p$ be a prime, and assume that $\lambda \in \F_p^*$ has multiplicative order $d \ge 3$ in $\F_p^*$. The continued fraction expansion $\{ \lambda /p \}=[0;a_1,\ldots, a_n]$ satisfies
\[ \sum_{i=1}^n a_i \le 13.44 p^{1-1/\varphi (d)} . \]
\end{thm}

Combining Theorem~\ref{thm:CF}  with Lemma~\ref{lem:D and CF 2}
%%the inequality~\eqref{eq:D and CF} 
below, we immediately derive 
the following bound,  which establishes~\cite[Conjecture~7.1]{LoMu2}. 

\begin{cor}
\label{cor:D-single}  
 Let $p$ be a prime, and assume that $\lambda \in \mathbb{F}_p^*$ has multiplicative order $d \ge 3$ in $\F_p^*$. Then
\[ |\ss(\lambda, p)| \le 1.12 p^{1-1/\varphi (d)} . \]
\end{cor}

\begin{rem} 
\label{rem:limit}  
A quick computation shows that Theorem~\ref{thm:CF} and Corollary~\ref{cor:D-single} are non-trivial (in other words, $\sum_{i=1}^n a_i = o(p)$ and $\ss(\lambda,p) = o(p)$) in the full range $\varphi(d)=o\left(\log p\right)$, while~\eqref{LMbound} is only non-trivial in the range $\varphi(d) \ll \log p / \log \log p$. Moreover, if $p=2^d -1$ is a Mersenne prime, then 
by~\cite[Equation~(7.1)]{LoMu2} we have $\ss(2, 2^d-1) \sim p/24$,  showing that 
$\ss(\lambda, p) = o(p)$, and thus by  Lemma~\ref{lem:D and CF 2} we see that  $\sum_{i=1}^n a_i = o(p)$,  
%%see~\eqref{eq:D and CF}, 
cannot hold in the range $\varphi(d) \asymp \log p$.
\end{rem}

\subsection{Correlation of Dedekind sums}

We first generalise the formula~\eqref{eq:M & D} to higher moments, relating $M_{2k}^-(p,m)$ to certain correlations between Dedekind sums.  

\begin{thm}\label{thm-formulahigher}
Let $m\mid p-1$ be even.  Then, 
\[
M_{2k}^-(p,m)   = \frac{2\pi^{2k}}{p^k}  \sum_{\lambda \in \cG_m}  \sum_{t_1,\ldots ,t_{k-1}=1}^{p-1} \ss(t_1,p) \cdots  \ss(t_{k-1},p)  \ss\left(\lambda t_1\cdots  t_{k-1},p\right).
\]
\end{thm} 

It may be useful to compare this when $\cG_m=\left\{1\right\}$ with a result of Alkan~\cite[Theorem~2]{Alk} on moments of  $L(1,\chi)$ (rather than their absolute values) modulo a composite $q$. See also \cite{BR,Zag} for works in the same spirit. Let $\tau_k$ denote the $k$-fold divisor function, that is,
\[
\tau_k(n)=\sum_{\substack{n_1\cdots n_k=n \\ n_i \geq 1}} 1.
\] 
As usual, for $k=2$ we write $\tau= \tau_2$ for the usual divisor function. 

Let us recall that by~\cite[Theorem~2.1]{MuSh} we know that for any $\kappa <1$ we have 
\begin{equation}\label{asymp_Mk}
M_{2k}^-(p,m)  = a(k) + O(p^{-\kappa/\varphi(d)} + p^{-1/4+o(1)}),
\end{equation}
 where 
 \[a(k)=\sum_{n=1}^{+\infty} \frac{\tau_{k}^2(n)}{n^2}.\]

Combining Theorem~\ref{thm-formulahigher} with~\eqref{asymp_Mk}  in the trivial case of $m = p-1$, that is, for $\cG_m=\left\{1\right\}$,  implies 
\begin{equation}
\begin{split}
\label{eq:asymp_Dedekind}
   \sum_{t_1,\ldots ,t_{k-1}=1}^{p-1} \ss(t_1,p) \cdots  \ss(t_{k-1},p)   & \ss\left(t_1\cdots  t_{k-1},p\right)  \\
&  \qquad \qquad = \(\frac{a(k)}{2\pi^{2k}} +o(1)\) p^k.
 \end{split}
\end{equation}

In the general case, Theorem~\ref{thm-formulahigher} and the relations~\eqref{asymp_Mk} 
and~\eqref{eq:asymp_Dedekind}, yield the following result.

 \begin{cor}\label{cor:saving}
 Let $\cG_m\subseteq \F_p^*$ be  of index $m$ and  of odd order $d =(p-1)/m$ such that $\varphi(d)=o(\log p)$. Then,
\[
\sum_{\substack{\lambda \in \cG_m \\ \lambda \neq 1}}  \sum_{t_1,\ldots ,t_{k-1}=1}^{p-1} \ss(t_1,p) \cdots  \ss(t_{k-1},p)  \ss\left(\lambda  t_1\cdots  t_{k-1},p\right) = o(p^k).
\]
\end{cor}

In fact, we see from~\eqref{asymp_Mk} that the bound of Corollary~\ref{cor:saving} is of magnitude $O(p^{k-\kappa/\varphi(d)}+p^{k-1/4+o(1)})$ with any $\kappa <1$.

In the case $k=2$, it is possible to give an independent proof  of Corollary~\ref{cor:saving} using the pointwise bound of Corollary~\ref{cor:D-double}
below  on correlations, which might be of independent interest. 
As in the  case of $k=1$, we start with a result on the partial quotients of the continued fraction expansions of two multiplicatively related rational numbers. 

\begin{thm}\label{thm:CF-2} 
Let $p$ be a prime and $\lambda, t \in \F_p^*$, and assume that $\lambda$ has multiplicative order $d \ge 3$ in $\mathbb{F}_p^*$. Then the continued fraction expansions $\{ \lambda t/p \} = [0;a_1, \ldots, a_n]$ and $\{ t/p \} = [0;b_1, \ldots, b_{\ell}]$ satisfy
\[ \left( \sum_{i=1}^n a_i \right) \left( \sum_{j=1}^{\ell} b_j \right) \le 26.88 p^{2-1/\varphi (d)} . \]
\end{thm}

As before, combining Theorem~\ref{thm:CF-2} with 
 Lemma~\ref{lem:D and CF 2}
%% the inequality~\eqref{eq:D and CF} 
 below leads to the following bound.

%\ccg{Routine calculus shows that 
%\begin{align*}
%\max_{\substack{u, v  \ge 1\\ uv \le A}} \(\frac{1}{4} + \frac{1}{12} u\) \(\frac{1}{4} + \frac{1}{12} v\) & \le  \frac{1}{144} 
%\max_{1 \le u \le A^{1/2}} \( 3+u \) \(3+A/u\) \\
%&= \frac{1}{144} 
%\max_{1 \le u \le A^{1/2}} \( 9+3u +3A/u + A\) \\
%&= \frac{1}{144}  \( 12 +4A\). 
%\end{align*}
%
% As before, combining Theorem~\ref{thm:CF-2}  with the 
% inequality~\eqref{eq:D and CF} below,  which implies 
% \begin{align*}
% |\ss(\lambda t, p) \ss (t,p)| &\le \frac{1}{16} + \frac{1}{48} \sum_{i=1}^n a_i + \frac{1}{48} \sum_{j=1}^{\ell} b_j + \frac{1}{144} \left( \sum_{i=1}^n a_i \right)  \left( \sum_{j=1}^{\ell} b_j \right) \\ 
% &\le \left( \frac{1}{16} + \frac{26.88}{48} + \frac{26.88}{48} + \frac{26.88}{144} \right) p^{2-1/\varphi(d)},  
%\end{align*}
%Hence  combining Theorem~\ref{thm:CF-2}  with the inequality~\eqref{eq:D and CF} below,  
%(and using that $26.88/36 < 3/4$)
%we obtain the following bound. }

\begin{cor}
\label{cor:D-double}  
 Let $p$ be a prime and $\lambda, t \in \F_p^*$, and assume that $\lambda$ has multiplicative order $d \ge 3$ in $\F_p^*$. Then
\[ |\ss(\lambda t , p) \ss(t,p)| < \frac{3}{16} p^{2-1/\varphi (d)} . \]
\end{cor}

\begin{rem} The bounds of Theorem~\ref{thm:CF-2} and Corollary~\ref{cor:D-double} are non-trivial (in other words, $\sum_{i=1}^n a_i \sum_{j=1}^{\ell} b_j = o(p^2)$ and $\ss(\lambda t,p)\ss(t ,p) = o(p^2)$) in the range $\varphi(d)=o(\log p)$. \end{rem}

Next, for integers $k_1,k_2\bmod p$, we define the following averaged correlations of Dedekind sums
\[
\mathcal{S}_{k_1,k_2}(p)=\sum_{t=1}^{p-1}  \ss(k_1 t,p)\ss(k_2t,p).
\]
A direct application of the Cauchy--Schwarz inequality shows that 
\begin{equation}\label{eq:triv corr}  
\left\vert \mathcal{S}_{k_1,k_2}(p) \right\vert\leq \sum_{t=1}^{p-1}  \ss( t,p)^2 \ll p^2, 
\end{equation}
where we have used that a special case of~\cite[Theorem~1]{Zhang} gives the asymptotic formula 
\begin{equation}\label{eq:asymp s2}  
 \sum_{t=1}^{p-1}  \ss( t,p)^2  = \frac{5}{144}p^2 + O\(p^{1+o(1)}\).
\end{equation}
By  Theorem~\ref{thm-formulahigher} (with $k=2$) and~\eqref{eq:asymp s2}   we have 
\begin{align*}
M_{4}^-(p,m)  & =\frac{2\pi^4}{p^2}   \sum_{\lambda \in \cG_m}  \mathcal{S}_{\lambda,1}(p)\\
&=     \frac{5}{144}p^2 +   \frac{2\pi^4}{p^2}   \sum_{ \substack{\lambda \in \cG_m \\ \lambda \neq 1}}  \mathcal{S}_{\lambda,1}(p) + O\(p^{1+o(1)}\) .
\end{align*}     Corollary~\ref{cor:saving} shows that the Dedekind sums $\ss(\lambda t,p)$ and $\ss(t,p)$ do not correlate in a strong form when averaged over both $t \in \F_p^*$ and $\lambda \in \cG_m \setminus \{1\}$. Indeed, Corollary~\ref{cor:saving} improves the trivial bound  $dp^2$ by~\eqref{eq:triv corr} for $ \mathcal{S}_{\lambda,1}(p)$ on average over  $\lambda \in \cG_m \setminus \{1\}$, and gives  
\[
\sum_{\substack{\lambda \in \cG_m \\ \lambda \neq 1}} \mathcal{S}_{\lambda,1}(p)
 \ll p^{2-1/(2\varphi(d))} + p^{7/4+o(1)}.
\]  

On the other hand, we prove that for given integers $k_1,k_2\ge 1$ which are 
not too large, the Dedekind sums $\ss(k_1 t,p)$ and $\ss(k_2 t,p)$ have a large correlation when averaged over $t$ modulo $p$, thus the trivial bound~\eqref{eq:triv corr} is the best possible. 
This is directly related to the twisted fourth moment of $L$-functions over the full group of characters (see the formula~\eqref{eq:M4- Asymp} below).

\begin{thm}
\label{thm:CorrCorr} 
For any fixed $\varepsilon>0$ and arbitrary coprime integers $k_1\ge k_2\ge 1$ with 
$
k_1^3 k_2^2 \le p^{1-\varepsilon}
$
we have,  as $p\to \infty$,
\[ \mathcal{S}_{k_1,k_2}(p) =  \(1+o(1)\) \frac{p^2}{2\pi^4}  
\sum_{n=1}^{\infty}
\frac{\tau(k_1n)\tau(k_2n)}{k_1k_2 n^2}. \]
 \end{thm}

\section{Preliminaries}

\subsection{Dedekind sums and continued fractions}
We recall that by a result of Barkan~\cite[Equation~(13)]{Bar} and  Hickerson~\cite[Theorem~1]{Hick} we have 
the following connection between Dedekind sums and continued fractions. 

\begin{lemma}
\label{lem:D and CF}
Let $1 \le a < b$ be integers, and let $a/b=[0;a_1, \ldots, a_n]$ be the continued fraction expansion with $a_n>1$. Then
\[ \ss(a,b) = \frac{(-1)^n -1}{8} + \frac{1}{12} \left( \frac{a}{b} + (-1)^{n+1} [0;a_n, \ldots, a_1] + \sum_{i=1}^n (-1)^{i+1} a_i \right) . \]
\end{lemma}  

In fact, we only need the following upper bound.

\begin{lem}
\label{lem:D and CF 2}
Let $1 \le a < b$ be integers, and let $a/b=[0;a_1, \ldots, a_n]$ be the continued fraction expansion with $a_n>1$. Then
$$
%%\begin{equation}
%%\label{eq:D and CF}
|\ss(a,b)| \le \frac{1}{12} \sum_{i=1}^n a_i .
%%\end{equation}
$$
\end{lem}

\begin{proof} Note that $a/b$ and $[0;a_n,\ldots, a_1]$ both lie in the interval $[0,1]$. For even $n$, Lemma \ref{lem:D and CF} gives
\[ \begin{split} \ss(a,b) &= \frac{1}{12} \left( \frac{a}{b} - [0;a_n,\ldots, a_1] + \sum_{i=1}^n (-1)^{i+1} a_i \right) \\ &\le \frac{1}{12} \left( a_1 + a_3 + \cdots + a_{n-1} \right) , \end{split} \]
and
\[ \ss(a,b) \ge \frac{1}{12} \left( -a_2 -a_4- \cdots - a_n \right) , \]
which immediately imply the claim. For odd $n \ge 3$, we similarly 
deduce from Lemma~\ref{lem:D and CF} that
\[ \begin{split} \ss(a,b) &= - \frac{1}{4} + \frac{1}{12} \left( \frac{a}{b} + [0;a_n, \ldots, a_1] + \sum_{i=1}^n (-1)^{i+1} a_i \right) \\ &\le \frac{1}{12} \left( a_1 + a_3 + \cdots + a_n \right) , \end{split} \]
and
\[ \begin{split} \ss(a,b) &\ge - \frac{1}{4} + \frac{1}{12} \left( a_1 + a_3 -a_2-a_4-\cdots -a_{n-1} \right) \\ &\ge \frac{1}{12} \left( -1-a_2-a_4-\cdots -a_{n-1} \right) , \end{split} \]
and the claim follows. Finally, for $n=1$ the assumption $a_1 \ge 2$ implies
\[ \ss (a,b) = - \frac{1}{4} + \frac{1}{12} \left( \frac{1}{a_1} + \frac{1}{a_1} +a_1 \right) \in \left[ 0, \frac{a_1}{12} \right] , \]
as claimed.
\end{proof}

\subsection{Small solutions to linear congruences}

We need the following result.  A similar estimate has been shown in the proof of~\cite[Lemma~4.6]{LoMu1}, but here we give a simple proof which also produces better numerical constants.

\begin{lemma}
\label{lem:SmallSols}
Assume that $\lambda$ has multiplicative order $d \ge 3$ in $\mathbb{F}_p^*$. Then,
 \[
\min\{ |m|+|h|:~ h \lambda \equiv m \pmod p, \ (h, m) \ne (0,0)\} \ge  p^{1/\varphi (d)}. 
\]
% \[
%\min\{ (h^2+m^2)^{1/2} :~ h \lambda \equiv m \pmod p, \ (h, m) \ne (0,0)\} \ge \frac{1}{2} p^{1/\varphi (d)}. 
%\]
\end{lemma}  

\begin{proof} Clearly $\Phi_d(\lambda) = 0$ (in $\F_p$), where 
$\Phi_d(X) \in \Z[X]$ is the $d$-th cyclotomic polynomial. 

We now fix  an integer pair $(h, m) \ne (0,0)$ with $ h \lambda \equiv m \pmod p$. 
Since $d \ge 3$,  the polynomial $\Phi_d$ has no rational roots (as all its roots are roots 
of unity of order $d$).  Therefore, $A = h^{\varphi(d)} \Phi_d(m/h)$ is a non-zero integer.

 On the other hand, 
\[
A = h^{\varphi(d)} \Phi_d(m/h)  \equiv h^{\varphi(d)} \Phi_d(\lambda) \equiv 0 \pmod p.
\]
Hence
\[
p \le |A| = \prod_{\substack{j=1\\\gcd(j,d)=1}}^d \left| m - h e^{2 \pi i j/d} \right| \le \(|m|+|h|\)^{\varphi(d)} ,
\]
which concludes the proof. 
\end{proof} 

 \section{Proofs of results on Dedekind sums}
 
\subsection{Proof of Theorem~\ref{thm:CF}}
Let $1 \le h \le p-1$ be any integer, and let $h \lambda \equiv m \pmod{p}$ with $-(p-1)/2 < m \le (p-1)/2$. 
Then $\left\| h \lambda /p \right\| = |m|/p$. By Lemma~\ref{lem:SmallSols}, we have
\[ p^{1/\varphi (d)} \le |h| + |m| = |h| + p \left\| \frac{h \lambda}{p} \right\| . \]
We can certainly assume  that $1 \le \lambda \le p-1$.
Let $p_i/q_i=[0;a_1, \ldots, a_i]$ denote the convergents to $ \lambda /p = [0;a_1, \ldots, a_n]$. 
From the general properties of continued fractions, see, for 
example,~\cite[Theorem~164]{HaWr}, we have 
\[
\left| \frac{\lambda}{p} - \frac{p_{i-1}}{q_{i-1}}\right| \le 
\frac{1}{q_{i-1} q_i } = \frac{1}{q_{i-1} \(a_iq_{i-1} + q_{i-2}\)} \le \frac{1}{a_iq_{i-1}^2} ,
\] 
hence $\| q_{i-1} \lambda /p \| \le 1/(a_i q_{i-1})$. In particular, for any $1 \le i \le n$,
\begin{equation}
\label{eq:aiqi}
p^{1/\varphi (d)} \le q_{i-1} + p \left\| \frac{q_{i-1} \lambda}{p} \right\| \le q_{i-1} + \frac{p}{a_i q_{i-1}} .
\end{equation}

Consider the following two cases. If $q_{i-1} \le p^{1/\varphi (d)}/2$, then~\eqref{eq:aiqi} implies that 
\[a_i \le \frac{2 p^{1-1/\varphi (d)}}{q_{i-1}} .
\] 
If $q_{i-1} > p^{1/\varphi (d)}/2$, then 
\[
a_i \le q_i / q_{i-1} \le p/q_{i-1}.
\] 
Therefore
\begin{align*}
 \sum_{i=1}^n a_i &= \sum_{\substack{i=1 \\ q_{i-1} \le p^{1/\varphi (d)} / 2}}^n a_i + \sum_{\substack{i=1 \\ q_{i-1} > p^{1/\varphi (d)} / 2}}^n a_i \\
 &\le \sum_{i=1}^n \frac{2p^{1-1/\varphi (d)}}{q_{i-1}} + \sum_{\substack{i=1 \\ q_{i-1} > p^{1/\varphi (d)} / 2}}^n \frac{p}{q_{i-1}}   .
  \end{align*}
Using the inequality
\[
q_{j+i-1} \ge q_j F_i,
\]
where $F_i$ are the Fibonacci numbers, which is immediate from the identity $q_{i+2} = a_{i+2}q_{i+1}+q_i$, we derive
\[ \sum_{i=1}^n a_i \le \left( 4 \sum_{i=1}^{\infty} \frac{1}{F_i} \right) p^{1-1/\varphi(d)} = 13.4395\ldots p^{1-1/\varphi(d)} , \]
which concludes the proof.

 \subsection{Proof of Theorem~\ref{thm-formulahigher}}
We recall the formula from~\cite[Equation~(3)]{Wal}
\begin{equation}\label{eq:ortho}  \sum_{r,t=1}^{p-1} \<\frac{r}{p}\>  \<\frac{t}{p}\> \overline{\chi}(r)\chi(t) = \begin{cases} \frac{p}{\pi^2} \vert L(1,\chi)\vert^2, & \text{ if } \chi(-1)=-1, \\
0, & \text{ otherwise}. \end{cases} \end{equation}

Summing over characters and using orthogonality relations, we derive
\begin{align*}
\frac{p^{k}}{\pi^{2k}}  \sum_{\chi\in   \cX_{p,m}^{-}}&\vert L(1,\chi)\vert^{2k} \\
& =   \sum_{\chi\in  \cX_{p,m}} \sum_{\substack{x_1,\ldots ,x_k=1 \\ y_1,\ldots ,y_k=1}}^{p-1} \prod_{i=1}^{k} \<\frac{x_i}{p}\>  \<\frac{y_i}{p}\> %%\prod_{i=1}^{k}
 \chi(x_i)\overline{\chi}(y_i) \\
& =   \#  \cX_{p,m} \sum_{\lambda \in \cG_m}
 \sum_{\substack{x_1,\ldots ,x_k=1 \\ y_1,\ldots ,y_k=1\\
 x_1\cdots x_k  \equiv \lambda y_1\cdots y_k \bmod p }}^{p-1} 
 \prod_{i=1}^{k}\<\frac{x_i}{p}\>  \<\frac{y_i}{p}\>.
\end{align*} 
For any $1 \leq i \leq k-1$, we make the change of variables $y_i = x_i t_i$. Hence, for every 
$\lambda \in \cG_m$  we obtain
\begin{align*}
 & \sum_{\substack{x_1,\ldots ,x_k=1 \\ y_1,\ldots ,y_k=1\\
 x_1\cdots x_k  \equiv \lambda y_1\cdots y_k \bmod p }}^{p-1}  \prod_{i=1}^{k} \<\frac{x_i}{p}\>  \<\frac{y_i}{p}\>  \\
 & \qquad = \sum_{\substack{x_1,\ldots,  x_{k} =1\\ t_1, \ldots ,t_{k-1},y_k=1\\ x_k \equiv \lambda t_1 \cdots t_{k-1} y_k \bmod p}}^{p-1}  \<\frac{y_k}{p}\>  \<\frac{x_k}{p}\>\prod_{i=1}^{k-1} \<\frac{x_i}{p}\>  \<\frac{x_i t_i}{p}\> \\
 & \qquad  = \sum_{\substack{x_1,\ldots,  x_{k-1} =1\\ t_1, \ldots ,t_{k-1},y_k=1}}^{p-1}  \<\frac{y_k}{p}\>  \<\frac{ \lambda t_1 \cdots t_{k-1} y_k}{p}\>\prod_{i=1}^{k-1} \<\frac{x_i}{p}\>  \<\frac{x_i t_i}{p}\> \\
 & \qquad =    \sum_{t_1,\ldots ,t_{k-1}=1}^{p-1} \ss(t_1,p) \cdots  \ss(t_{k-1},p)  \ss\left(\lambda  t_1\cdots  t_{k-1},p\right)  .
\end{align*} 
Thus, in accordance with~\eqref{eq:M & D},  we derive
  \begin{align*}
& \frac{p^{k}}{\pi^{2k}}  \sum_{\chi\in  \cX_{p,m}^{-}}\vert L(1,\chi)\vert^{2k}  \\
 &\quad =  \#  \cX_{p,m} \sum_{\lambda \in \cG_m}  \sum_{t_1,\ldots ,t_{k-1}=1}^{p-1} \ss(t_1,p) \cdots  \ss(t_{k-1},p)  \ss\left(\lambda t_1\cdots  t_{k-1},p\right) , 
  \end{align*}
 which  finishes the proof of Theorem~\ref{thm-formulahigher}.

\subsection{Proof of Theorem~\ref{thm:CF-2}}
\label{sec:CorrDed} 
 Let $h_1, h_2 \in [1,p-1]$ be any integers, and let $h_1 \lambda t \equiv m_1 \pmod{p}$ and $h_2 t \equiv m_2 \pmod{p}$ with $m_1, m_2 \in (-(p-1)/2, (p-1)/2]$. Then $\| h_1 \lambda t /p \| = |m_1|/p$ and $\| h_2 t /p \| = |m_2|/p$. Observe that $h_2 m_1 - \lambda h_1 m_2 \equiv 0 \pmod{p}$. Therefore by Lemma \ref{lem:SmallSols},
\[ p^{1/\varphi(d)} \le  |h_2 m_1| + |h_1 m_2| = |h_2| p \left\| \frac{h_1 \lambda t}{p} \right\| + |h_1| p \left\| \frac{h_2 t}{p} \right\| . \]
Let $p_i/q_i=[0;a_1,\ldots, a_i]$ denote the convergents to $\{ \lambda t/p \}$, and $\widetilde{p}_j / \widetilde{q}_j = [0;b_1, \ldots, b_j]$ the convergents to $\{ t/p \}$. In particular, for any $1 \le i \le n$ and $1 \le j \le \ell$, we derive
\[ p^{1/\varphi(d)} \le \widetilde{q}_{j-1} p \left\| \frac{q_{i-1} \lambda t}{p} \right\| + q_{i-1} p \left\| \frac{\widetilde{q}_{j-1} t}{p} \right\| \le 2 p \max \left\{ \frac{\widetilde{q}_{j-1}}{a_i q_{i-1}} , \, \frac{q_{i-1}}{b_j \widetilde{q}_{j-1}} \right\} . \]
Taking the reciprocals yields
\[ 2 p^{1-1/\varphi (d)} \ge \min \left\{ \frac{a_i q_{i-1}}{\widetilde{q}_{j-1}} ,\,  \frac{b_j \widetilde{q}_{j-1}}{q_{i-1}} \right\} . \]
Letting 
\[\cH = \left\{ (i,j) \in [1,n] \times [1,\ell] :~\frac{a_i q_{i-1}}{\widetilde{q}_{j-1}} \le \frac{b_j \widetilde{q}_{j-1}}{q_{i-1}} \right\},
\] 
we have
  \begin{align*}  \sum_{i=1}^n a_i   \sum_{j=1}^{\ell} b_j  
&= \sum_{(i,j) \in \cH} a_i b_j + \sum_{(i,j) \not\in \cH} a_i b_j \\ &\le \sum_{(i,j) \in \cH} 2 p^{1-1/\varphi (d)} \frac{\widetilde{q}_{j-1}}{q_{i-1}} b_j + \sum_{(i,j) \not\in \cH} a_i 2 p^{1-1/\varphi (d)} \frac{q_{i-1}}{\widetilde{q}_{j-1}} \\ &\le 2 p^{1-1/\varphi (d)} \left( \sum_{i=1}^n \frac{1}{q_{i-1}} \sum_{j=1}^{\ell} b_j \widetilde{q}_{j-1} + \sum_{i=1}^n a_i q_{i-1} \sum_{j=1}^{\ell} \frac{1}{\widetilde{q}_{j-1}} \right)  . 
 \end{align*}
   Using that 
\[
\sum_{i=1}^{n} a_i q_{i-1} = q_n + q_{n-1} -1 \le 2p \quad \text{and}\quad 
\sum_{j=1}^{\ell} b_j \widetilde{q}_{j-1} = \widetilde{q}_{\ell} + \widetilde{q}_{\ell-1} -1 \le 2p, \]
we now derive 
\[  \sum_{i=1}^n a_i   \sum_{j=1}^{\ell} b_j \le 
 8  p^{2-1/\varphi (d)}  \sum_{i=1}^{\infty} \frac{1}{F_i} 
= 26.8790\ldots  p^{2-1/\varphi (d)},
\] 
where, as before, $F_i$ are the Fibonacci numbers,
which concludes the proof.

\subsection{Proof of Theorem~\ref{thm:CorrCorr}}

For any integers $k_1,k_2 \geq 1$  let us define the following twisted fourth moments
\begin{align*}
& M_4^-(p;k_1,k_2)=\frac{2}{p-1}\sum_{\chi \in \cX_p^-} \chi(k_1)\overline{\chi}(k_2) \left\vert L(1,\chi)\right\vert^4,\\
& M_{4}(p;k_1,k_2)= \frac{1}{p-1}\sum_{\chi \in \cX_p^*} 
\chi(k_1)\overline{\chi}(k_2) \left\vert L(1,\chi)\right\vert^4,
\end{align*}
where $\cX_p^* = \cX_p\setminus \{\chi_0\}$ denotes the set of all non-principal characters modulo $p$. Then,  by~\cite[Theorem~1.1]{Lee19} 
uniformly over  $k_1,k_2 \geq 1$,  we have the following asymptotic formula  
\begin{equation}\label{asymptwistedfourth} 
M_{4}(p;k_1,k_2)= \sum_{\substack{n \in \mathbb{Z} \\ {n \neq 0}}} \frac{\tau(k_1n)\tau(k_2n)}{k_1k_2 n^2}
+ O\(\sqrt{k_1+k_2} p^{-1/2 + o(1)}  \). 
\end{equation}
It follows from the proof of~\cite[Theorem~1.1]{Lee19} that $M_{4}^-(p;k_1,k_2)$ satisifies the same 
asymptotic formula~\eqref{asymptwistedfourth}  as $M_{4}(p;k_1,k_2)$. That is, 
\begin{equation}\label{eq:M4- Asymp} 
M_{4}^-(p;k_1,k_2)=\sum_{\substack{n \in \mathbb{Z} \\ {n \neq 0}}}\frac{\tau(k_1n)\tau(k_2n)}{k_1k_2 n^2} + O\(\sqrt{k_1+k_2} p^{-1/2 + o(1)}  \). 
\end{equation}

We now recall~\eqref{eq:ortho}. Hence, summing over characters, we derive, similarly as in the proof 
of Theorem~\ref{thm-formulahigher}, 
\begin{align*}
\frac{p^2}{\pi^4}   \sum_{\chi\in   \cX_{p}^{-}} &  \chi(k_1)\overline{\chi}(k_2) \left\vert L(1,\chi)\right\vert^4   \\\\
& =  (p-1) \sum_{\substack{r,s,t,u=1 \\ k_1rt \equiv k_2su  \bmod p}}^{p-1}\<\frac{r}{p}\>  \<\frac{t}{p}\> \<\frac{s}{p}\>  \<\frac{u}{p}\>.
\end{align*} 
Making the change of variables  $ t  \to k_2  st$,  we have
 \begin{align*}
 \frac{p^2}{\pi^4}   \sum_{\chi\in   \cX_{p}^{-}}& \chi(k_1)\overline{\chi}(k_2) \left\vert L(1,\chi)\right\vert^4 \\
& =   (p-1)  \sum_{\substack{r,s,t,u =1 \\ k_1 rt \equiv   u   \bmod p}}^{p-1}  \<\frac{r}{p}\>  \<\frac{s}{p}\>
 \<\frac{ k_2st }{p}\>  \<\frac{u}{p}\> \\
 & =  (p-1) \sum_{r,s,t=1}^{p-1} \<\frac{ r}{p}\>  \<\frac{s}{p}\>
 \<\frac{ k_2st }{p}\>   \<\frac{k_1rt}{p}\>. \\
\end{align*} 
Next, we change $ s  \to k_2^{-1} s$, and we derive  
 \begin{align*}
 \frac{p^2}{\pi^4} & \sum_{\chi\in   \cX_{p}^{-}} \chi(k_1)\overline{\chi}(k_2) \left\vert L(1,\chi)\right\vert^4 \\
 & = (p-1) \sum_{r,s,t=1}^{p-1} \<\frac{ r}{p}\>  \<\frac{k_2^{-1} s}{p}\>
 \<\frac{st }{p}\>   \<\frac{k_1rt }{p}\>  \\
 &  =  (p-1)\sum_{t=1}^{p-1}  \ss(k_1t ,p)\ss(k_2t,p)   .
\end{align*}

Hence, we obtain 
\[
\sum_{t=1}^{p-1}  \ss(tk_1,p)\ss(tk_2,p) = \frac{p^2}{2\pi^4}  \cdot  \frac{2}{p-1} \sum_{\chi\in \cX_p^-}\chi(k_1)\overline{\chi}(k_2) \vert L(1,\chi)\vert^4  ,
\]
which together with~\eqref{eq:M4- Asymp}  and using that $k_1 \ge k_2$, 
concludes the proof.

\section{Comments}
%This can also be explained using continued fractions
Note that when we bound the Dedekind sum from above in terms of the sum
of partial quotients, we ignore the cancellation in the alternating sum of
partial quotients in Lemma~\ref{lem:D and CF}. However, we might speculate that there is no 
such cancellation in the extremal cases. Indeed, let us look at the aforementioned example~\eqref{eq:Bad b}.
%% $\frac{a^d-1}{a-1}$ mentioned in the Introduction. 
For a fixed prime $d$, we expect $p=\frac{a^d-1}{a-1}$ to be prime for infinitely many $a$. Notice that the continued fraction expansion is $a/p=[0;a_1,a_2]$ with 
\[
a_1=\frac{a^{d-1}-1}{a-1} = a^{d-2}+\dots+a+1 \mand a_2=a.
\] 
Thus 
\[
a_1+a_2  \sim  a^{d-2} \sim p^{(d-2)/(d-1)} =p^{1-1/\varphi(d)}
\] 
and $\ss(a,p) \sim \frac{1}{12}p^{1-1/\varphi(d)}$ as $a \to \infty$ by Lemma~\ref{lem:D and CF}. In particular, the best possible constant in the upper bound in Theorem~\ref{thm:CF} is between $1$ and $13.44$. Numerical evidence in~\cite[Section~7]{LoMu2} suggests that this example is indeed extremal for Dedekind sums. We thus conjecture that the best possible constant in the upper bound in Corollary~\ref{cor:D-single} is $1/12$ instead of $1.12$.

%Note that in this extremal case we get the upper bound in Theorem \ref{thm:CF} with constant $1$ instead of $13.44$. We might conjecture (supported by numerical computations \cite[Section $7$]{LoMu2}) that extremal examples comes from  numbers with one (or a few) partial quotients dominating the others. This makes plausible to expect that the best possible constant in Theorem \ref{thm:CF} is $1$ and in Corollary \ref{cor:D-single} is $1/12$. 

%In order to obtain the bound of Corollary \ref{cor:D-single}, 
%Notice that the continued fraction expansion is $a/p=[0;a_1,a_2]$ with $a_1 = (a^{d-1}-1)/(a-1) = a^{d-2}+\dots+a+1$
%and $a_2=a$. Thus $a_1+a_2 \sim p^{1-1/(d-1)}$  and \eqref{asymp-Mersenne} follows by Lemma \ref{lem:D and CF}. We might conjecture (supported by numerical computations \cite[Section $7$]{LoMu2}) that extremal examples comes from  numbers with one (or a few) partial quotients dominating the others and that the best possible constant in Theorem \ref{thm:CF} is $1$ and in Corollary \ref{cor:D-single} is $1/12$. 
%Note that in this extremal case we get the upper
%bound in Theorem \ref{thm:CF} with constant $1$ instead of $13.44$. 
%Note that when we bound the Dedekind sum from above in terms of the sum
%of partial quotients, we ignore the cancellation (absent in the extremal cases) in the alternating sum of
%partial quotients in Lemma \ref{lem:D and CF}.
%However, in view of this example it is
%plausible that in extremal examples one (or a few) partial quotients
%dominate all the others, and in fact there is basically no cancellation.

 \section*{Acknowledgement}

The authors  would like to thank Sandro Bettin for pointing out the reference~\cite{Bettin} which may be useful to obtain an alternative proof of  Corollary~\ref{cor:saving}  for $k=2$.  The authors are also grateful to 
St{\'e}phane Louboutin for useful comments and  to the anonymous referee for pointing out a way to improve Lemma \ref{lem:SmallSols} leading to a better constant in Corollary \ref{cor:D-single}.

During the preparation of this work B.B. was supported by the Austrian Science Fund (FWF) project M 3260-N, M.M.  %% was supported
 by the Ministero della Istruzione e della Ricerca 
Young Researchers Program Rita Levi Montalcini and I.S. %% was supported in part
 by the  
Australian Research Council Grants  DP230100530 and DP230100534.

\end{document}